\documentclass{amsart} 

\newtheorem{theorem}{Theorem}[section]
\newtheorem{lemma}[theorem]{Lemma}
\newtheorem{proposition}[theorem]{Proposition}

\theoremstyle{definition}

\newtheorem{question}[theorem]{Question}

\newcommand{\R}{\mathbb{R}}

\newcommand{\lw}[1]{\smash{\lower0.6ex\hbox{#1}}}

\usepackage{graphicx}

\begin{document}


\title[Crossing changes, Delta moves, and sharp moves on welded knots]
{Crossing changes, Delta moves, and sharp moves on welded knots} 


\author{Shin SATOH}
 
\address{Department of Mathematics, Kobe University, 
Rokkodai-cho 1-1, Nada-ku, Kobe 657-8501, Japan}

\email{shin@math.kobe-u.ac.jp}

Dedicated to Professor Yasutaka Nakanishi on the occasion of his 60th birthday


\renewcommand{\thefootnote}{\fnsymbol{footnote}}





\maketitle


\begin{abstract} 
We prove that the crossing changes, Delta moves, and sharp moves 
are unknotting operations on welded knots. 
\end{abstract}


\section{Introduction}\label{sec1}

The virtual knots \cite{Kau3} and welded knots \cite{FRR} 
are two extensions of classical knots in the Euclidian $3$-space. 
In classical knot theory, 
invariants and local moves play important roles 
from the algebraic and geometric viewpoints. 
Several algebraic invariants of classical knots are extended to 
those of virtual or welded knots. 
For example, the Jones polynomial is an invariant of a virtual knot 
but not that of a welded knot, 
and the knot group and knot quandle are invariants of 
a virtual knot and a welded knot both. 
As for local moves on virtual knots, 
there are many results in relation to finite type invariants. 
In particular, a replacement of a classical crossing with a virtual crossing 
is used in \cite{GPV}. 

In this paper, we consider three kinds of classical local moves 
called the {\it crossing change}, the {\it Delta move}, and the {\it sharp move} 
as shown in Figure~\ref{fig001}. 
These local moves are known as unknotting operations for classical knots \cite{Mur, MN}. 
On the other hand, the crossing change on a virtual knot 
is not an unknotting operation; 
for example, the value of the Miyazawa (or arrow) polynomial \cite{DK,Miy}
at $A=1$ detects the non-triviality of a virtual knot up to crossing changes. 
Since a Delta move and a sharp move are presented by 
crossing changes, 
neither of the moves is an unknotting operation for virtual knots 
(cf.~\cite{ST}). 
The aim of this paper is to prove that the three local moves 
are unknotting operations for welded knots 
in the following sense.

\begin{figure}[htb]
\begin{center}
\includegraphics[bb=0 0 346 51]{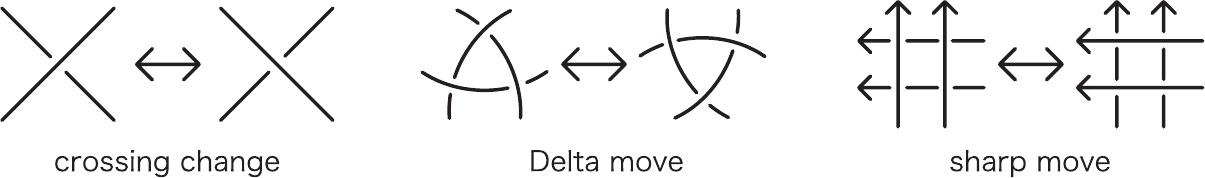}
\caption{}
\label{fig001}
\end{center}
\end{figure}

\begin{theorem}\label{thm11}
For any diagram $D$ of a welded knot $K$, 
there is a diagram $D'$ such that 
\begin{itemize}
\item[{\rm (i)}] 
$D'$ is obtained from by the crossing changes at some classical crossings, and 
\item[{\rm (ii)}] 
$D'$ presents the trivial welded knot. 
\end{itemize}
\end{theorem} 

\begin{theorem}\label{thm12}
For any welded knot $K$, 
there is a finite sequence of welded knots 
$K=K_0,K_1,\dots,K_n$ such that 
\begin{itemize}
\item[{\rm (i)}] 
$K_i$ is obtained from $K_{i-1}$ by a Delta move 
$(i=1,2,\dots,n)$, and 
\item[{\rm (ii)}] 
$K_n$ is the trivial welded knot. 
\end{itemize}
\end{theorem} 

\begin{theorem}\label{thm13}
For any welded knot $K$, 
there is a finite sequence of welded knots 
$K=K_0,K_1,\dots,K_n$ such that 
\begin{itemize}
\item[{\rm (i)}] 
$K_i$ is obtained from $K_{i-1}$ by a sharp move 
$(i=1,2,\dots,n)$, and 
\item[{\rm (ii)}] 
$K_n$ is the trivial welded knot. 
\end{itemize}
\end{theorem} 

This paper is organized as follows. 
In Section~\ref{sec2}, 
we prove that any descending diagram 
presents a trivial welded knot, 
which induces Theorem~\ref{thm11}. 
In Sections~\ref{sec3} and \ref{sec4}, 
we prove that a replacement of a classical crossing with a welded crossing 
is accomplished by Delta moves and sharp moves, respectively, 
which induces Theorems~\ref{thm12} and \ref{thm13}.


\section{Crossing changes}\label{sec2} 

A {\it welded knot diagram} 
is a circle immersed in the plane ${\R}^2$ 
with transverse double points which are divided into two classes called 
{\it classical crossings} and {\it welded crossings}. 
A classical crossing has over/under-information 
such that a small segment is removed from one of the paths 
intersecting at the crossing, 
and a welded crossing is indicated by putting a small circle on it. 
See Figure~\ref{fig002}. 
A welded knot diagram is called {\it trivial} 
if it has no classical and welded crossings, 
that is, it is an embedding of a circle in ${\R}^2$.

\begin{figure}[htb]
\begin{center}
\includegraphics[bb=0 0 173 48]{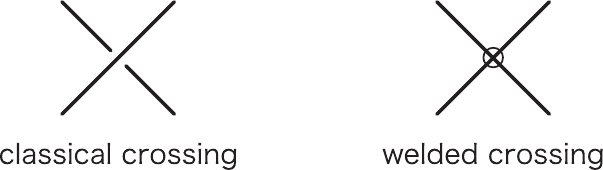}
\caption{}
\label{fig002}
\end{center}
\end{figure}

We consider eight kinds of local moves on welded knot diagrams called {\it Reidemeister moves} 
as shown in Figure~\ref{fig003}. 
The first three moves C1--C3 are classical Reidemeister moves. 
The next three moves V1--V3 are obtained from C1--C3 by replacing all the classical crossings 
with welded ones. 
The moves V4 and W are also obtained from C3 by replacing two or one classical crossing(s) 
with welded one(s), respectively. 
Here, the arc with two welded crossings passes the classical crossing for V4, 
and the arc with two classical crossings passes over (not under) the welded crossing for W.

\begin{figure}[htb]
\begin{center}
\includegraphics[bb=0 0 343 123]{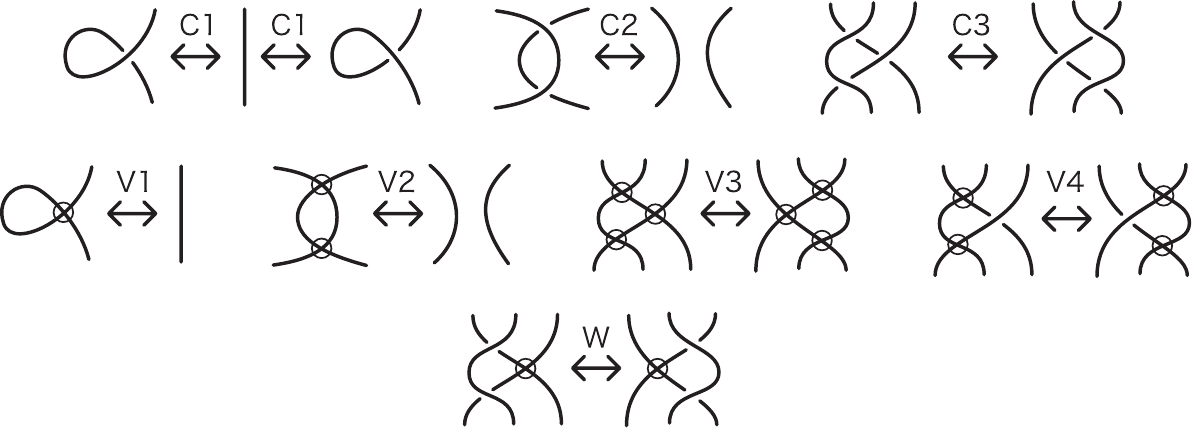}
\caption{}
\label{fig003}
\end{center}
\end{figure}

We say that two welded knot diagrams $D$ and $D'$ are {\it equivalent} 
if there is a finite sequence of welded knot diagrams 
$D=D_0,D_1,\dots,D_n=D'$ 
such that $D_i$ is obtained from $D_{i-1}$ 
by performing a Reidemeister move C1--C3, V1--V4, or W on $D_{i-1}$ 
$(i=1,2,\dots,n)$. 
A {\it welded knot} is an equivalence class of welded knot diagrams 
under these Reidemeister moves. 
A welded knot is called {\it trivial} if it is presented by 
a trivial diagram.

A {\it Gauss diagram} is the union of a circle 
and $n$ signed and oriented chords for some $n\geq 0$ 
connecting $n$ pairs of points on the circle. 
Let $D$ be a welded knot diagram with $n$ classical crossings. 
The Gauss diagram $G=G(D)$ associated with $D$ 
is defined to the the union of a circle covering $D$ 
and $n$ chords connecting the preimages of classical crossings. 
Each chord has the sign derived from that of the classical crossing, 
and is oriented from the over-crossing to the under-crossing.

We consider four kinds of moves on Gauss diagrams 
corresponding to moves C1, C2, C3, and W on welded knot diagrams. 
We use the same notations to indicate the moves on Gauss diagrams 
as those on welded knot diagrams. 
The left and middle of Figure~\ref{fig004} 
show C1 and W on a Gauss diagram, respectively, 
where $\varepsilon$ and $\varepsilon'$ are any signs. 
A move C1 removes or adds a chord 
whose endpoints are adjacent to each other. 
Such a chord is called {\it trivial}. 
Also, a move W changes the positions of adjacent initial endpoints of two chords 
regardless of the signs. 
We remark that a crossing change on a welded knot diagram 
corresponds to the change of sign and orientation of the chord 
simultaneously. 
See the right of the figure.

\begin{figure}[htb]
\begin{center}
\includegraphics[bb=0 0 334 37]{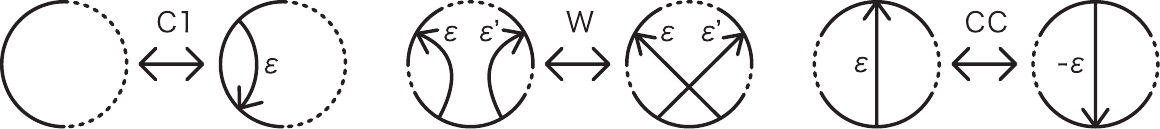}
\caption{}
\label{fig004}
\end{center}
\end{figure}

It is known that two welded knot diagrams $D$ and $D'$ 
define the same Gauss diagram 
$G(D)=G(D')$ if and only if $D$ and $D'$ are related by 
a finite sequence of virtual Reidemeister moves V1--V4 \cite{GPV}.

\begin{lemma}\label{lem21}
Let $x$ be a classical crossing of a welded knot diagram $D$. 
We divide $D$ into two closed paths by cutting $D$ at $c$. 
Suppose that one of the obtained paths contains no under-crossing except $x$. 
Let $E$ be the welded knot diagram obtained from $D$ 
by replacing $x$ with a welded crossing. 
Then $D$ is related to $E$ by a finite sequence of 
{\rm C1}, {\rm V1--V4}, and {\rm W}. 
\end{lemma}

\begin{proof}
Let $\alpha$ be the closed path at $x$ on $D$ 
which contains no under-crossings. 
Since the path 
on the Gauss diagrams $G(D)$ corresponding to $\alpha$ 
contains no terminal endpoints of chords, 
we shrink the chord corresponding to $x$ to be trivial 
by applying W repeatedly. 
Then the chord is removed by C1 so that we obtain 
the Gauss diagram $G(E)$ of $E$. 
See Figure~\ref{fig005}. 
Therefore, $D$ and $E$ are related by C1 and W 
up to V1--V4. 
\end{proof}

\begin{figure}[htb]
\begin{center}
\includegraphics[bb=0 0 240 110]{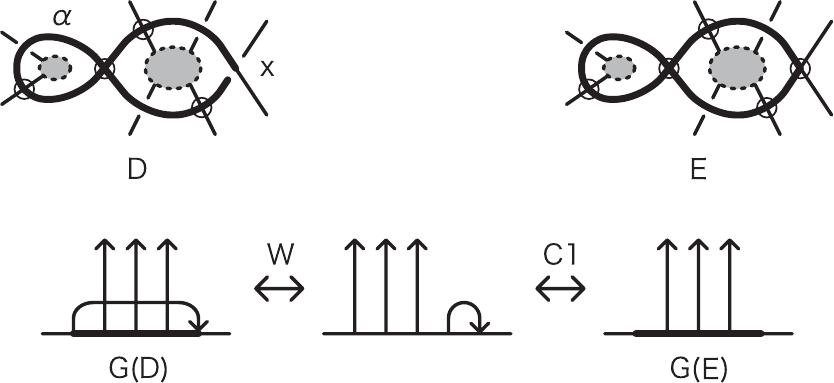}
\caption{}
\label{fig005}
\end{center}
\end{figure}

We say that a welded knot diagram $D$ is {\it descending} 
if there is a base point and an orientation of $D$ such that 
walking along $D$ from the base point 
with respect to the orientation, 
we meet the over-crossing for the first time and the under-crossing 
for the second time at every classical crossing.

\begin{proposition}\label{prop22} 
Any descending diagram $D$ is related to the trivial diagram 
by a finite sequence of {\rm C1}, {\rm V1--V4}, and {\rm W}. 
\end{proposition}

\begin{proof}
Let $x$ be the classical crossing of $D$ 
such that $x$ is the first under-crossing 
while walking along $D$ from the base point 
according to the orientation of $D$. 
Since $x$ satisfies the condition in Lemma~\ref{lem21}, 
we can replace $x$ with a welded crossing 
by C1, V1--V4, and W. 
Since the obtained diagram is descending again, 
by repeating this process, 
$D$ is deformed into the diagram 
where all crossings are welded. 
Such a welded knot diagram is related to the trivial one 
by V1--V4. 
\end{proof}

\begin{proof}[Proof of {\rm Theorem~\ref{thm11}}]
By Proposition~\ref{prop22}, 
it is sufficient to perform crossing changes on $D$ 
so that the obtained diagram is descending. 
\end{proof}

Let ${\rm c}(K)$ denote the minimal number of classical crossings 
for all diagrams of a welded knot $K$.

\begin{lemma}\label{lem23}
Any non-trivial welded knot $K$ satisfies ${\rm c}(K)\geq 3$. 
\end{lemma}

\begin{proof}
It is not difficult to see that 
if a Gauss diagram has at most two chords, 
then the chords can be removed by C1 and W. 
\end{proof}

For a welded knot diagram $D$, 
we denote by $u(D)$ the minimal number of classical crossings of $D$ 
for which we perform the crossing changes to obtain a diagram 
presenting the trivial welded knot. 
The number $u(D)$ is well-defined by Theorem~\ref{thm11}. 
The {\it unknotting number} of a welded knot $K$ 
is the minimal number of $u(D)$ for all diagrams $D$ presenting $K$, 
and denoted by ${\rm u}(K)$. 
The following is a generalization of the well-known result for a classical knot 
(cf.~\cite{Oza}).

\begin{proposition}\label{prop24}
Any non-trivial welded knot $K$ satisfies 
$\displaystyle{{\rm u}(K)\leq \frac{{\rm c}(K)-1}{2}}$. 
\end{proposition}

\begin{proof}
Let $D$ be a welded knot diagram of $K$ with $c(D)={\rm c}(K)$, 
and $x$ a classical crossing of $D$. 
By Lemma~\ref{lem23}, 
we have $c(D)\geq 3$. 

We choose a pair of points $p_1$ and $p_2$ 
on the over-path at $x$ from one side of $x$ to the other. 
Let $S_i$ $(i=1,2)$ be the set of classical crossings of $D$ 
such that we perform crossing changes at $S_i$ on $D$ 
to obtain the descending diagram with the base point $p_i$ 
and the the orientation from $p_i$ to $p_j$ $(j\ne i)$. 
Since $S_1\cap S_2=\emptyset$ and 
$|S_1|+|S_2|=c(D)-1$, 
it follows by Proposition~\ref{prop22} that 
$${\rm u}(K)\leq u(D)\leq \frac{c(D)-1}{2}=\frac{{\rm c}(K)-1}{2}.$$
\end{proof}


\section{Delta moves}\label{sec3}

\begin{proposition}\label{prop31}
Let $x$ be a classical crossing of a welded knot diagram $D$, 
and $E$ the welded knot diagram obtained from $D$ 
by the crossing change at $x$. 
Then $D$ and $E$ are related by a finite sequence of 
Reidemeister moves and Delta moves. 
\end{proposition}

\begin{proof}
Fix an orientation of $D$. 
Then $D$ is regarded as a band sum of 
a positive or negative Hopf link diagram $H$ 
as shown in Figure~\ref{fig006}. 
This replacement is realized by C2. 
We remark that each band is on the left hand side of 
the attaching arc on $E$.

\begin{figure}[htb]
\begin{center}
\includegraphics[bb=0 0 264 49]{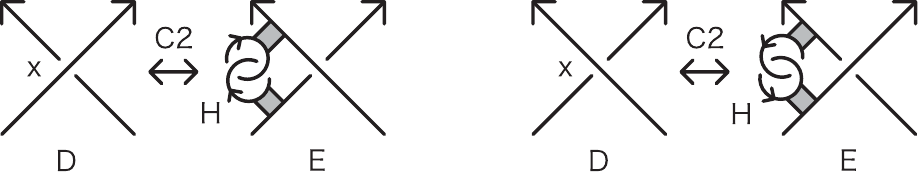}
\caption{}
\label{fig006}
\end{center}
\end{figure}

We slide the attaching arcs of the bands along $E$ 
so that they are adjacent to each other on $E$. 
In the top row of Figure~\ref{fig007}, 
an attaching arc passes a crossing of $E$ 
with making a pair of classical or welded crossings. 
In the bottom of the figure, 
an attaching arc passes an intersection of $E$ and a band 
with making a quadruple of classical or welded crossings. 
The deformations are accomplished by C2 and V2.

\begin{figure}[htb]
\begin{center}
\includegraphics[bb=0 0 325 91]{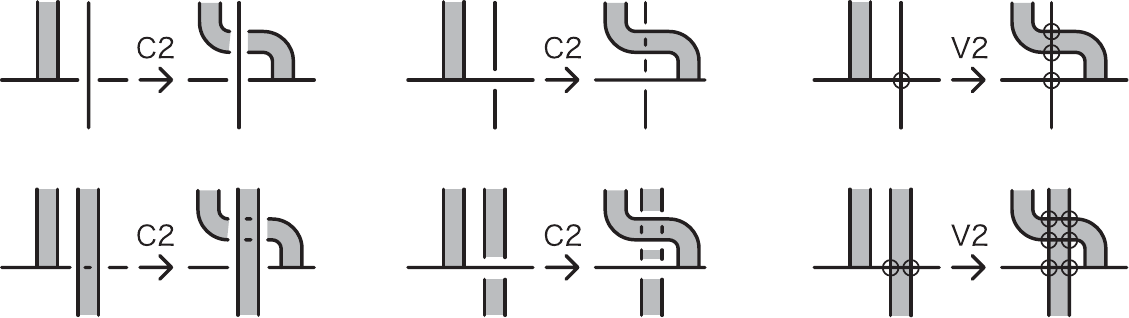}
\caption{}
\label{fig007}
\end{center}
\end{figure}

The Hopf link diagram $H$ can slide 
anywhere between the bands. 
In Figure~\ref{fig008}, 
$H$ passes an intersection of $E$ and a band 
by performing C2, C3, V2, and V4. 
The case that $H$ passes an intersection between bands 
is similarly proved by replacing the horizontal segment of $E$ in the figure 
with a band.

\begin{figure}[htb]
\begin{center}
\includegraphics[bb=0 0 334 46]{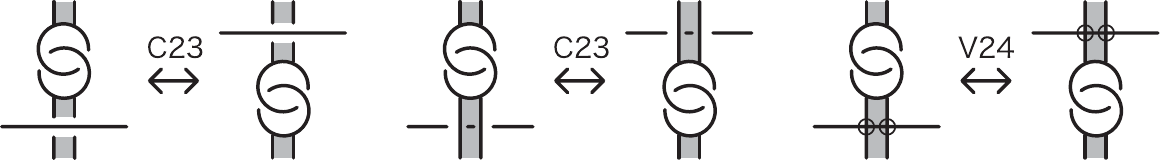}
\caption{}
\label{fig008}
\end{center}
\end{figure}

A {\it banded} Reidemeister move is a local deformation 
obtained from an original Reidemeister move as in Figure~\ref{fig003} 
by replacing some strings with bands. 
It is easy to see that any banded Reidemeister move except C1 is 
accomplished by some original Reidemeister moves. 
To proceed with the proof of Proposition~\ref{prop31}, 
we prepare the following classical deformations.

\begin{lemma}[\cite{TY}]\label{lem32}
For a band sum of $E$ and $H$, we have the following. 
\begin{itemize}
\item[{\rm (i)}] 
A banded Reidemeister move {\rm C1} is realized by 
classical Reidemeister moves and a Delta move 
up to a slide of $H$. 
\item[{\rm (ii)}] 
A crossing change between $E$ and a band 
or between bands is realized by classical Reidemeister moves and 
a Delta move up to a slide of $H$. 
\end{itemize}
\end{lemma}

\begin{proof}
We slide $H$ near the portion where the modification will be applied. 
Then a banded C1 move and a crossing change between $E$ and a band 
are accomplished by classical Reidemeister moves and Delta moves 
as shown in Figures~\ref{fig009}. 
A crossing change between bands 
is similarly proved by duplicating the horizontal segment of $E$ in the figure. 
\end{proof}

\begin{figure}[htb]
\begin{center}
\includegraphics[bb=0 0 321 118]{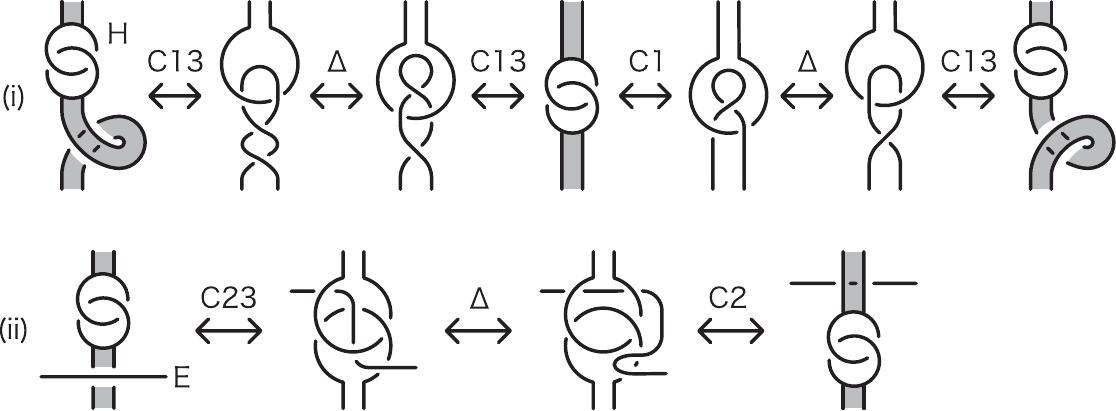}
\caption{}
\label{fig009}
\end{center}
\end{figure}

Now we continue the proof. 
By Lemma~\ref{lem32}(ii), 
we perform crossing changes so that the union of two bands are descending, 
that is, 
walking from one attaching arc on $E$ to the other, we meet 
\begin{itemize}
\item[(i)] 
a pair of over-crossings at every classical intersection between $E$ and a band, and 
\item[(ii)] 
a quadruple of over-crossings for the first time 
and a quadruple of under-crossings for the second time 
at every classical intersection between bands. 
\end{itemize}
See the left of Figure~\ref{fig010}.

\begin{figure}[htb]
\begin{center}
\includegraphics[bb=0 0 347 65]{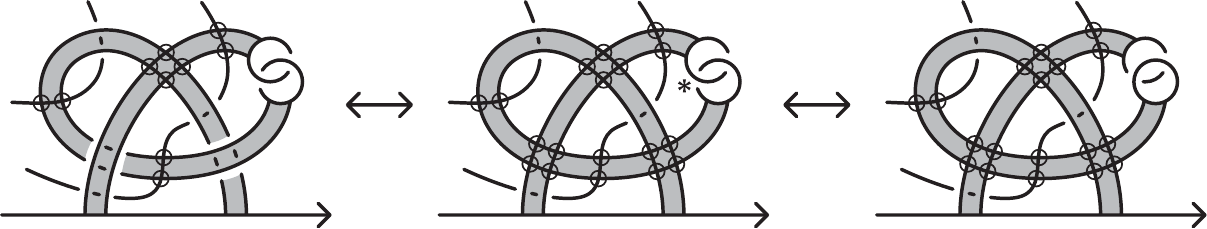}
\caption{}
\label{fig010}
\end{center}
\end{figure}

We use the same technique as in the proof of Lemma~\ref{lem21} 
for the cores of the bands. 
By performing banded Reidemeister moves C1 and W by Lemma~\ref{lem32}(i), 
all classical intersections between bands are replaced with welded ones 
as shown in the middle of the figure. 
Since one of the classical crossing of $H$ 
satisfies the condition in Lemma~\ref{lem21}, 
we perform a welding of the crossing and the inverse of another 
so that $H$ becomes unlinked. 
See the right of the figure. 
Since the bands are removed by C2 and V2, 
$D$ is equivalent to $E$ up to Delta moves. 
This completes the proof of 
Proposition~\ref{prop31}. 
\end{proof}

\begin{proof}[Proof of {\rm Theorem~\ref{thm12}}]
By Proposition~\ref{prop22}, 
it is sufficient to perform crossing changes on $D$ 
so that the obtained diagram is descending. 
Such crossing changes are accomplished 
by Reidemeister moves and Delta moves 
by Proposition~\ref{prop31}. 
\end{proof}

In classical knot theory, 
a single Delta move necessarily changes the knot type \cite{Oka}. 
On the other hand, there is a pair of diagrams 
of the same welded knot which are related by a single Delta move. 
The two diagrams as shown in Figure~\ref{fig011} 
present the trivial welded knot.

\begin{figure}[htb]
\begin{center}
\includegraphics[bb=0 0 260 55]{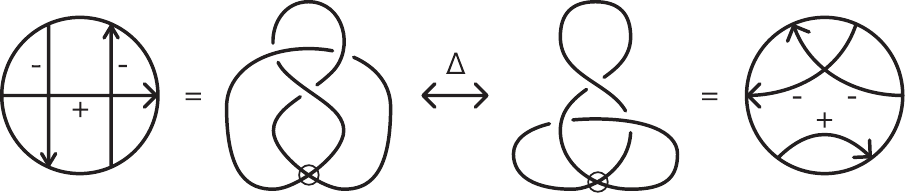}
\caption{}
\label{fig011}
\end{center}
\end{figure}

By a similar argument to the proof of Proposition~\ref{prop31}, 
any welded {\it link} diagram is transformed into the one such that  
all the self-crossings of the same component are welded 
by a finite sequence of Reidemeister moves and Delta moves. 
It is known in \cite{MN} that two classical links are related by a finite sequence of 
classical Reidemeister moves and Delta moves if and only if 
their pairwise linking numbers coincide.

\begin{question}\label{qst33}
Can we detect whether two welded links 
are related by Reidemeister moves and Delta moves 
by some algebraic invariants?
\end{question}


\section{Sharp moves}\label{sec4}

\begin{proposition}\label{prop41}
Let $x$ be a classical crossing of a welded knot diagram $D$, 
and $E$ the welded knot diagram obtained from $D$ 
by the crossing change at $x$. 
Then $D$ and $E$ are related by a finite sequence of 
Reidemeister moves and sharp moves. 
\end{proposition}

\begin{proof}
The story of the proof is exactly the same as that of Proposition~\ref{prop31}. 
It is sufficient to prove the following 
analogous to Lemma~\ref{lem32}. 
\end{proof}

\begin{lemma}\label{lem42}
For a band sum of $E$ and $H$, 
we have the following. 
\begin{itemize}
\item[{\rm (i)}] 
A banded Reidemeister move {\rm C1} is realized by 
classical Reidemeister moves and sharp moves 
up to a slide of $H$. 
\item[{\rm (ii)}] 
A crossing change between $E$ and a band 
or between bands is realized by classical Reidemeister moves 
sharp move up to a slide of $H$. 
\end{itemize}
\end{lemma}

\begin{proof}
We  consider four kinds of local moves, 
called a {\it pass move} \cite{Erl, Kau1}, a {\it $t_4$ move},  
a {\it $\overline{t_4}$ move} \cite{Prz}, and a {\it $\Gamma$ move} \cite{Kau1, Kau2} 
as shown in Figure~\ref{fig012}. 
It is well-known that 
\begin{itemize}
\item
a pass move is realized by classical Reidemeister moves and sharp moves \cite{MN}, 
\item
a $t_4$ move is realized by classical Reidemeister moves and a sharp move (cf.~\cite{NN}),  
\item
a $\overline{t_4}$ move is realized by classical Reidemeister moves and a pass move 
(and hence, sharp moves), and 
\item
a $\Gamma$ move is realized by classical Reidemeister moves and a pass move 
(and hence, sharp moves) \cite{Kau1}. 
\end{itemize}

\begin{figure}[htb]
\begin{center}
\includegraphics[bb=0 0 349 37]{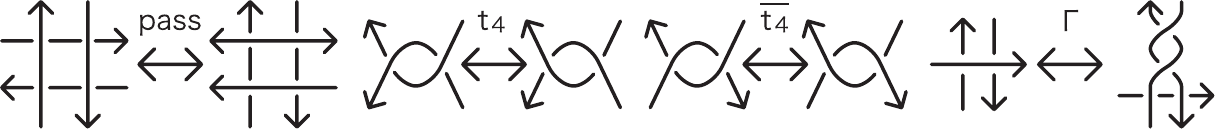}
\caption{}
\label{fig012}
\end{center}
\end{figure}

Now we slide $H$ near the the portion 
where the modification will be applied. 
Then a banded C1 move and a crossing change between 
$E$ and a band are accomplished 
by classical Reidemeister moves and sharp moves 
as shown in Figure~\ref{fig013}. 
A crossing change between bands is 
exactly the same as a pass move. 
\end{proof}

\begin{figure}[htb]
\begin{center}
\includegraphics[bb=0 0 272 109]{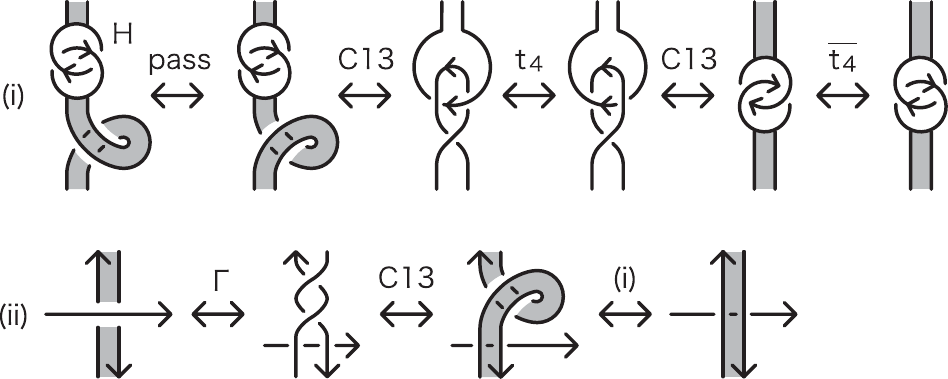}
\caption{}
\label{fig013}
\end{center}
\end{figure}

\begin{proof}[Proof of {\rm Theorem~\ref{thm13}}]
By Proposition~\ref{prop22}, 
it is sufficient to perform crossing changes on $D$ 
so that the obtained diagram is descending. 
Such crossing changes are accomplished 
by Reidemeister moves and sharp moves 
by Proposition~\ref{prop41}. 
\end{proof}

In classical knot theory, 
a single sharp move necessarily changes the knot type \cite{Mur}. 
On the other hand, there is a pair of diagrams 
of the same welded knot which are related by a single sharp move. 
The two diagrams as shown in Figure~\ref{fig014} 
present the trivial welded knot.

\begin{figure}[htb]
\begin{center}
\includegraphics[bb=0 0 298 46]{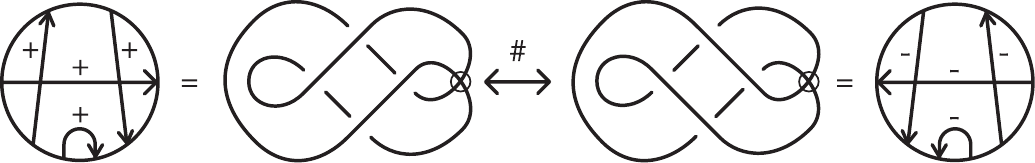}
\caption{}
\label{fig014}
\end{center}
\end{figure}

By a similar argument to the proof of Proposition~\ref{prop41}, 
any welded {\it link} diagram is transformed into the one such that  
all the self-crossings of the same component are welded 
by a finite sequence of Reidemeister moves and sharp moves. 
The necessary and sufficient condition are known 
for classical links to be related by classical Reidemeister moves 
and sharp moves in terms of their linking numbers \cite{MN}.

\begin{question}\label{qst43}
Can we detect whether two welded links 
are related by Reidemeister moves and sharp moves 
by some algebraic invariants?
\end{question}

It is known that any classical knot is equivalent to 
the trivial knot or trefoil knot up to pass moves \cite{Kau1}.

\begin{question}\label{qst44}
Is the set of equivalence classes of welded knots 
up to pass moves finite, or infinite?
\end{question}

\section*{Achknowledgments}
The author is partially supported by 
JSPS KAKENHI Grant Number 25400090.
He would like to thank Professors Takuji Nakamura and Yasutaka Nakanishi 
for helpful suggestions.


\end{document}